\documentclass{amsart}
\usepackage{amsthm,hyperref}
\usepackage{amsmath,amssymb,MnSymbol,hyperref,cite,tikz}

\numberwithin{equation}{section}
\newtheorem{thm}{Theorem}[section]

\newtheorem{cor}[thm]{Corollary}
\newtheorem{lem}[thm]{Lemma}

\theoremstyle{definition}
\newtheorem{defn}[thm]{Definition}

\newtheorem{rem}[thm]{Remark}

\begin{document}
\title{\textbf{Cliques in the union of $C_4$-free graphs}}
\author{Abeer Othman and Eli Berger}
\address{Department of Mathematics, University of Haifa, Mount Carmel, Haifa 31905, Israel}
\email{abeer.othman@gmail.com}
\address{Department of Mathematics, University of Haifa, Mount Carmel, Haifa 31905, Israel}
\email{berger@math.haifa.ac.il}
\begin{abstract}
Let $B$ and $R$ be two simple graphs with vertex set $V$, and let $G(B,R)$ be the simple graph with vertex set $V$, in which two vertices are adjacent if they are adjacent in at least one of $B$ and $R$. We prove that if $B$ and $R$ are two $C_4$-free graphs on the same vertex set $V$ and $G(B,R)$ is the complete graph, then there exists an $B$-clique $X$, an $R$-clique $Y$ and a clique $Z$ in $B$ and $R$, such that $V=X\cup Y\cup Z$. Further, if $x\in Z$ then $x$ is one of the vertices of some double $C_5$ in $G(B,R)$. In particular, if also $G(B,R)$ does not contains a double $C_5$, then $V$ is obedient. We obtain that if $B$ and $R$ are $C_4$-free graphs then $\omega(G(B,R))\leq \omega(B)+\omega(R)+\frac{1}{2}\min(\omega(B),\omega(R))$ and $\omega(G(B,R))\leq \omega(B)+\omega(R)+\omega(H(B,R))$ where $H(B,R)$ is the simple graph with vertex set $V$, in which two vertices are adjacent if they are adjacent in $B$ and $R$.
\end{abstract}
\maketitle

\section{Introduction}
All graphs in this article are simple in the sense that they do not have double edges or loops. Let $B$ (for ``Blue'') and $R$ (for ``Red'') be two graphs on the same vertex set $V$. Denote by $G(B,R)$ (respectively, $H(B,R)$) the graph with vertex set $V$, in which two vertices are adjacent if they are adjacent in $B$ or (respectively, and) in $R$. Recall that a \textbf{clique} in a graph is a set of pairwise adjacent vertices. If $G(B,R)$ is the complete graph, then it is naturally to ask: Does $V$ is the union of a clique in $B$ and a clique in $R$?. Easier examples show that the answer is ``No''. For example, if $B=C_5$ and $R$ is its complement, then $\omega(B)+\omega(R)=4<|V|=5$, where $\omega(G)$ is the maximal size of a clique in $G$. In particular, $V$ is not the union of a clique in $B$ and a clique in $R$. Some results showed that under some additional assumptions on the graphs $B$ and $R$, we obtain that $V$ is the union of a clique in $B$ and a clique in $R$. We begin to summarize them. First, we need the following definitions.

\begin{defn}
We say that a subset $U\subseteq V$ is \textbf{obedient} if there exist an $R$-clique $X$ and an $B$-clique $Y$ such that $U=X\cup Y$.
\end{defn}

\begin{defn}
A graph $G$ is called \textbf{chordal} if each of its cycles of four or more nodes has a \textbf{chord}, which is an edge joining two nodes that are not adjacent in the cycle. In other words, $G$ is $C_k$\textbf{-free} for $k\geq 4$, i.e., $G$ has no induced cycles of length at least four.
\end{defn}

In \cite{gy}, Gy{\'a}rf{\'a}s and Lehel proved the following theorem.

\begin{thm}[Theorem 3 of \cite{gy}]\label{8}
Let $B$ and $R$ be two graphs on the same vertex set $V$ and assume that $G(B,R)$ is the complete graph. If $B$ and $R$ are $C_k$-free for $k=4$ and $k=5$, then $V$ is obedient.
\end{thm}

The methods used in \cite{gy} are combinatorial. Using topological methods, Berger \cite{berger} proved the following theorem.

\begin{thm}[\cite{berger}]\label{9}
Let $B$ and $R$ be two graphs on the same vertex set and assume that $G(B,R)$ is the complete graph. If $B$ is chordal and $R$ is $C_4$-free, then $V$ is obedient.
\end{thm}

Recently, Aharoni, Berger, Chudnovsky and Ziani generalized Theorems \ref{8} and \ref{9} as follow.

\begin{thm}[Theorem 1.7 of \cite{cliques}]\label{10}
Let $B$ and $R$ be $C_4$-free graphs with vertex set $V$ and suppose that $R$ is also $C_5$-free. If $G(B,R)$ is the complete graph, then $V$ is obedient.
\end{thm}

Before we continue. We need the following definition.

\begin{defn}
We say that $G(B,R)$ contains a \textbf{double $C_5$} if there is a set $X$ of five vertices such that $B|X$ and $R|X$ are both induced $C_5$ in $B$ and $R$, where $B|X$ (or $R|X$) denotes the subgraph of $B$ (or $R$) induced by $X$.
\end{defn}

It is clear that if $B$ or $R$ is $C_5$-free, then $G(B,R)$ does not contains a double $C_5$. Very recently, Gy{\'a}rf{\'a}s and Lehel generalized Theorem \ref{10} as follow.

\begin{thm}[Theorem 3 of \cite{gy2}]\label{11}
Let $B$ and $R$ be $C_4$-free graphs with vertex set $V$, and suppose that $G(B,R)$ does not contains a double $C_5$. If $G(B,R)$ is the complete graph, then $V$ is obedient.
\end{thm}

In section 2, we start with some basic definitions and prove the following theorem which generalize Lemma 3.4 of \cite{cliques}.

\begin{thm}\label{12}
Let $B$ and $R$ be two graphs with vertex set $V$. Assume that the following hold:
\begin{itemize}
  \item $B$ and $R$ are $C_4$-free graphs.
  \item $G(B,R)$ does not contains a double $C_5$.
  \item $G(B,R)$ is the complete graph.
  \item Every proper subset $U\subset V$ is obedient.
\end{itemize}
If $v_1-v_2-v_3-v_4-v_5-v_1$ is an induced $C_5$ in $R$, where $v_1v_2, v_3v_4\in E(R\setminus B)$, then $V$ is obedient.
\end{thm}

In section 3, we use Theorem \ref{12} to give another proof of Theorem \ref{11}, which we found independently. In section 4, we generalize Theorem \ref{11} by proving the following.

\begin{thm}\label{13}
Let $B$ and $R$ be two $C_4$-free graphs on the same vertex set $V$ and assume that $G(B,R)$ is the complete graph. Then there exists an $B$-clique $X$, an $R$-clique $Y$ and a clique $Z$ in $B$ and $R$, such that $V=X\cup Y\cup Z$. Further, if $x\in Z$ then $x$ is one of the vertices of some double $C_5$ in $G(B,R)$.
\end{thm}

As a corollary of Theorem \ref{13}, we obtain the following theorems.

\begin{thm}
If $B$ and $R$ are two $C_4$-free graphs on the same vertex set $V$, then $$\omega(G(B,R))\leq \omega(B)+\omega(R)+\omega(H(B,R)).$$
\end{thm}

\begin{thm}
If $B$ and $R$ are two $C_4$-free graphs on the same vertex set $V$ then $$\omega(G(B,R))\leq \omega(B)+\omega(R)+\frac{1}{2}\min(\omega(B),\omega(R)).$$
\end{thm}
\section{Obedient Sets}
The main result of this section is Theorem \ref{2} which proved in \cite{cliques} when $B$ is $C_5$-free (see Lemma 3.4 of \cite{cliques}). Using a similar argument, we generalize it when $G(B,R)$ does not contains a double $C_5$. First, we recall some basic definitions. Let $G=(V,E)$ be a graph.

\begin{defn}
An \textbf{induced subgraph} on a subset $S$ of $V$, denoted by $G|S$, is a graph whose vertex set is $S$ and whose edge set is $\{uv~|~u,v\in S~\mathrm{and}~uv\in E\}$.
\end{defn}

\begin{defn}
For two disjoint subsets $X$ and $Y$ of $V$, we say that $X$ is \textbf{$G$-complete (anticomplete}) to $Y$ if every vertex of $X$ is adjacent (non-adjacent) to every vertex of $Y$.
\end{defn}

Let $B$ and $R$ be two graphs with vertex set $V$. We denote by $B\setminus R$ the graph with vertex set $V$ such that two vertices are adjacent in $B\setminus R$ if and only if they are adjacent in $B$ and non-adjacent in $R$.

\begin{defn}
\begin{itemize}
  \item A set $C\subset V$ is a \textbf{cutset} if there exist disjoint nonempty $P, Q\subset V$ such that $V\setminus C= P\cup Q$ and $P$ is anticomplete to $Q$ in $G$.
  \item A set $C\subset V$ is a \textbf{clique cutset} if it is a cutset and $C$ is a clique of $G$.
  \item A \textbf{weak clique cutset} in $B$ is a clique $C$ of $B$ that is a cutset in $B\setminus R$.
\end{itemize}
\end{defn}
\begin{defn}
For a graph $H$ we say that \textbf{$G$ \textbf{contains} $H$} if some induced subgraph of $G$ is isomorphic to $H$. We say that $G$ is \textbf{$H$-free} if $G$ is not contains $H$.
\end{defn}

Before we prove Theorem \ref{2}, we need the following useful theorem.
\begin{thm}\label{1}\emph{(Lemma 3.2 of \cite{cliques})}\\
Let $B$ and $R$ be $C_4$-free graphs with vertex set $V$, and assume that every proper subset $U\subset V$ is obedient. Assume also that $G(B,R)$ is a complete graph. If there is a weak clique cutset in $B$ (or in $R$), then $V$ is obedient.
\end{thm}

\begin{thm}\label{2}
Let $B$ and $R$ be two graphs with vertex set $V$. Assume that the following hold:
\begin{itemize}
  \item $B$ and $R$ are $C_4$-free graphs.
  \item $G(B,R)$ does not contains a double $C_5$.
  \item $G(B,R)$ is the complete graph.
  \item Every proper subset $U\subset V$ is obedient.
\end{itemize}
If $v_1-v_2-v_3-v_4-v_5-v_1$ is an induced $C_5$ in $R$, where $v_1v_2, v_3v_4\in E(R\setminus B)$, then $V$ is obedient.
\end{thm}

\begin{proof}
Since $v_1-v_2-v_3-v_4-v_5-v_1$ is an induced $C_5$, it follows that every edge in the cycle $v_1-v_3-v_5-v_2-v_4-v_1$ belongs to $E(B\setminus R)$. If $v_2v_3$ is a blue edge, then $v_1-v_4-v_2-v_3-v_1$ is an induced $C_4$ in $B$ and this contradict the assumption of the theorem. We obtain that $v_2v_3\in E(R\setminus B)$. Since $G(B,R)$ does not contains a double $C_5$, it follows that at least one of the edges $v_1v_5$, $v_5v_4$ belongs to $E(B)$. We may assume that $v_1v_5\in E(B)$. Consider the cycle $v_1-v_4-v_2-v_5-v_1$. If $v_4v_5$ is not blue, then this cycle is an induced $C_4$ in $B$, a contradiction. It follows that $v_4v_5$ is a blue edge. We conclude that $\{v_1,v_2,v_3,v_4,v_5\}$ is an obedient set, where $\{v_1,v_4,v_5\}$ is the blue clique and $\{v_2,v_3\}$ is the red clique. If $V=\{v_1,v_2,v_3,v_4,v_5\}$, then we are done. Let $u\in V\setminus\{v_1,v_2,v_3,v_4,v_5\}$.\\\\
\textbf{\emph{Claim:}} $u$ is $R$-complete to $\{v_2,v_3\}$ or $B$-complete to $\{v_1,v_4,v_5\}$.\\\\
\emph{Proof of the claim:}
If $u$ is $R$ complete to $\{v_2,v_3\}$, then we are done. Assume that $u$ is not $R$-complete to $\{v_2,v_3\}$ and whithout loss of generality, assume that $uv_2\in E(B\setminus R)$. We show that in this condition $u$ is $B$-complete to $\{v_1,v_4,v_5\}$. Assume in contrary that $uv_1\in E(R\setminus B)$. Since $R$ is $C_4$-free, it follows that $uv_3\in E(B\setminus R)$. For otherwise we have that $u-v_3-v_2-v_1-u$ is an induced $C_4$ in $R$. Since $B$ is $C_4$-free, it follows that $uv_4\in E(R\setminus B)$. For otherwise we have that $u-v_4-v_1-v_3-u$ is an induced $C_4$ in $B$. It follows that $v_4-v_2-u-v_3-v_1$ is a double $C_5$ in $G(B,R)$, a contradiction. So $uv_1$ is a blue edge.\\
Since $B$ is $C_4$-free, it follows that $uv_4\in E(B)$. For otherwise we obtain that $uv_4\in E(R\setminus B)$ and we have that $u-v_2-v_4-v_1-u$ is an induced $C_4$ in $B$.\\
Since $B$ is $C_4$-free, it follows that $uv_5\in E(B)$. For otherwise we have that $u-v_2-v_5-v_1-u$ is an induced $C_4$ in $B$. We conclude that $u$ is $B$-complete to $\{v_1,v_4,v_5\}$. Thus, we proved the claim.

Let $A$ be the set of vertices in $V\setminus\{v_1,\dots,v_5\}$ that are not $R$-complete to $\{v_2,v_3\}$. Then $A$ is $B$-complete to $\{v_1,v_4,v_5\}$.\\\\
\textbf{\emph{Claim:}} $A$ is an $B$-clique.\\\\
\emph{Proof of the claim:}
Let $u,v\in A$ and assume that $uv\in E(R\setminus B)$. Since $u$ is not $R$-complete to $\{v_2,v_3\}$, we may assume that $uv_2\in E(B\setminus R)$. We show that $vv_2\in E(R)$. If $vv_2\notin E(R)$, then $v-v_2-u-v_1-v$ is an induced $C_4$ in $B$, a contradiction. We conclude that $vv_2\in E(R)$ and so $vv_3\notin E(R)$ because $v\in A$. Similarly we obtain that $uv_3\in E(R)$. But now $u-v-v_2-v_3-u$ is a $C_4$ in $R$, a contradiction. This proves the claim that $A$ is a $B$-clique.

By the definition of $A$ and the pervious claim we conclude that that $A\cup \{v_1,v_4,v_5\}$ is a $B$-clique.
Let $Z=V\setminus (A\cup \{v_1,v_4,v_5\})$. If $Z=\{v_2,v_3\}$ then $V$ is obedient. Assume that $Z\setminus \{v_2,v_3\}\neq \emptyset$. Note that every vertex in $Z\setminus \{v_2,v_3\}$ is $R$-complete to $\{v_2,v_3\}$ and so $Z\setminus \{v_2,v_3\}$ is anticomplete to $\{v_2,v_3\}$ in the graph $B\setminus R$. It follows that $A\cup \{v_1,v_4,v_5\}$ is a weak clique cutset in $B$. By Theorem \ref{1}, $V$ is obedient.
\end{proof}

\section{Structures without double $C_5$}
In this section, we give another proof of the recent result by Gy{\'a}rf{\'a}s and Lehel, which we found independently.
\begin{thm}\label{3}
Let $B$ and $R$ be $C_4$-free graphs with vertex set $V$, and assume that $G(B,R)$ is the complete graph. If $G(B,R)$ does not contains a double $C_5$, then $V$ is obedient.
\end{thm}

\begin{proof}
We use a similar argument to that in \cite{gy} (see the proof of Theorem 3). We prove the Theorem by induction on $|V|$. The theorem hold for $|V|=1$ or $|V|=2$. Assume that it is true for $|V|=1,2,\dots,n$ and let $|V|=n+1$. Let $p\in V$. By the induction hypothesis $V\setminus \{p\}$ is obedient, namely there exist $X,Y\subseteq V$ such that $V\setminus \{p\}=X\cup Y$, $X$ is a red clique and $Y$ is a blue clique. We define the following sets \begin{center}
$X_b=\{a\in X~|~ap\in E(B\setminus R)\}$ and $Y_r=\{a\in Y~|~ap\in E(R\setminus B)\}$.
\end{center}
We may assume that $X$ and $Y$ is chosen with $|X_b|+|Y_r|$ minimum. If $|X_b|=0$ or $|Y_r|=0$, then the Theorem holds. Assume that $|X_b|\neq0$ and $|Y_r|\neq0$.\\ Let $q\in Y_r$. We claim that in the graph $B\setminus R$, $q$ is connected to a vertex in $X_b$. If $q$ is connected in red to all vertices of $X$, then $X\cup \{q\}$ is a red clique and $V\setminus \{p\}=(X\cup\{q\})\cup Y\setminus \{q\}$ with $|(X\cup\{q\})_b|+|(Y\setminus \{q\})_r|< |X_b|+|Y_r|$, a contradiction to the minimality of $|X_b|+|Y_r|$. It follows that there is $r\in X$ such that $qr\in E(B\setminus R)$. If $r\in X_b$, then we are done. Suppose that $r\notin X_b$. So $pr$ is a red edge. Let $s\in X_b$. If $qs$ is a red edge then $p-q-s-r-p$ is an $C_4$ induced cycle in $R$, a contradiction. So $qs\in E(B\setminus R)$ and the claim follows. Similarly, for every $u\in X_b$ there exists $v\in Y_r$ such that $uv\in E(R\setminus B)$.\\
Let $x_1\in Y_r$. There is $x_2\in X_b$ such that $x_1x_2\in E(B\setminus R)$. Also there is $x_3\in Y_r$ such that $x_2x_3\in E(R\setminus B)$. We continue in this way and obtain an even cycle $x_1-x_2-x_3-x_4-\cdots$ such that $x_{2j-1}\in Y_r$, $x_{2j}\in X_b$, $x_{2j-1}x_{2j}\in E(B\setminus R)$ and $x_{2j}x_{2j+1}\in E(R\setminus B)$ for all $j\geq1$. Let $x_1-x_2-x_3-x_4-\cdots-x_k-x_1$ be the shortest even cycle that we can get in this way.\\\\
\textbf{\emph{Case 1:}} k=4.\\\\
Since $B$ is $C_4$-free, it follows that $x_2x_4\in E(R\setminus B)$. For otherwise we have that $x_1-x_2-x_4-x_3-x_1$ is an induced $C_4$ in $B$. Since $R$ is $C_4$-free, it follows that $x_1x_3\in E(B\setminus R)$. For otherwise we have that $x_1-x_4-x_2-x_3-x_1$ is an induced $C_4$ in $R$. It follows that $p-x_3-x_2-x_4-x_1-p$ is a double $C_5$ in $G(B,R)$, a contradiction.\\\\
\textbf{\emph{Case 2:}} $k>4$.\\\\
By the minimality of $k$ we obtain that $x_1x_4$ is an edge in $B$ and $R$. Since $R$ is $C_4$-free, it follows that $x_1x_3\in E(B\setminus R)$. For otherwise we have that $x_1-x_4-x_2-x_3-x_1$ is an induced $C_4$ in $R$. It follows that $x_2-x_3-p-x_1-x_4-x_2$ is an induced $C_5$ in $R$ with $x_2x_3,px_1\in E(R\setminus B)$. By Theorem \ref{2}, $V$ is obedient.
\end{proof}

\begin{rem}
Another way to prove Theorem \ref{3} is to follow the proof of Theorem 1.7 in \cite{cliques} and use Theorem \ref{2} in every place the authors used Lemma 3.4 of \cite{cliques}.
\end{rem}

\section{Structures with double $C_5$}
In this section, we prove that if $B$ and $R$ are two $C_4$-free graphs on the same vertex set $V$ and $G(B,R)$ is the complete graph, then there exists an $B$-clique $X$, an $R$-clique $Y$ and a clique $Z$ in $B$ and $R$, such that $V=X\cup Y\cup Z$. Further, if $x\in Z$ then $x$ is one of the vertices of some double $C_5$ in $G(B,R)$. This generalize the recent result by Gy{\'a}rf{\'a}s and Lehel (Theorem \ref{3}). We begin with the following lemmas.

\begin{lem}\label{4}
Let $B$ and $R$ be two $C_4$-free graphs and $C:~v_1-v_2-v_3-v_4-v_5-v_1$ be a double $C_5$ in $G(B,R)$ with blue edges and red diagonals. Assume that $G(B,R)$ is the complete graph and $x\in V\setminus \{v_1,\dots,v_5\}$. Then one of the following holds:
\begin{enumerate}
  \item $x$ is connected in $B$ to all the vertices of $C$.
  \item $x$ is connected in $R$ to all the vertices of $C$.
  \item there exists $i$ such that\begin{itemize}
                                    \item $xv_i$ is an edge in $B$ and $R$.
                                    \item In $B\setminus R$: $x$ is connected to the two neighbors of $v_i$ in $C$.
                                    \item In $R\setminus B$: $x$ is connected to the two neighbors of $v_i$ in $C$.\\In this case we say that $xv_i$ is the \textbf{shared edge }of $x$.
                                  \end{itemize}
\end{enumerate}
\end{lem}
\begin{proof}
Let $T=\{v_1,\dots,v_5,x\}$. If $\deg_{B|T}(x)=0$, then condition (2) holds. Assume that $\deg_{B|T}(x)=1$. Without loss of generality, assume that $xv_1\in E(B)$. Since $G(B,R)$ is the complete graph, it follows that $xv_2,\dots,xv_5\in E(R)$. Since $R$ is $C_4$-free, it follows that $xv_1\in E(R)$. For otherwise we have that $x-v_3-v_1-v_4-x$ is an induced $C_4$ in $R$. So $\deg_{R|T}(x)=5$ and condition (2) holds. Assume that $\deg_{B|T}(x)=2$. If $xv_i,xv_j\in E(B)$, where $v_iv_j\in E(R\setminus B)$, then $x-v_i-v_m-v_j-x$ is an induced $C_4$ in $B$, where $v_m$ is the common neighbor of $v_i$ and $v_j$ in the cycle $v_1-v_2-v_3-v_4-v_5-v_1$, a contradiction. It follows that $x$ is connected to two successive vertices of the cycle $v_1-v_2-v_3-v_4-v_5-v_1$. Without loss of generality, assume that $xv_1,xv_2\in E(B)$. It follows that $xv_3,xv_4,xv_5\in E(R)$. Since $R$ is $C_4$-free, it follows that $xv_1\in E(R)$. For otherwise we have that $x-v_3-v_1-v_4-x$ is an induced $C_4$ in $R$. Since $R$ is $C_4$-free, it follows that $xv_2\in E(R)$. For otherwise we have that $x-v_4-v_2-v_5-x$ is an induced $C_4$ in $R$. So $\deg_{R|T}(x)=5$ and condition (2) holds. Assume that $\deg_{B|T}(x)=3$. A similar argument shows that $x$ is connected to three successive vertices of the cycle.  Without loss of generality, assume that $xv_1,xv_2,xv_3\in E(B)$. It follows that $xv_4,xv_5\in E(R)$. Since $R$ is $C_4$-free, it follows that $xv_2\in E(R)$. For otherwise we have that $x-v_4-v_2-v_5-x$ is an induced $C_4$ in $R$. If $xv_1\notin E(R)$ and $xv_3\notin E(R)$, then condition (3) holds. If $xv_1\in E(R)$, then $xv_3\in E(R)$. For otherwise, we have that $x-v_1-v_3-v_5-x$ is an induced $C_4$ in $R$. So condition (2) holds. Similarly, if $xv_3\in E(R)$, then condition (2) holds. Assume that $\deg_{B|T}(x)=4$. Without loss of generality, assume that $xv_1,xv_2,xv_3,xv_4\in E(B)$. Since $B$ is $C_4$-free, it follows that $xv_5\in E(B)$. For otherwise we have that $x-v_1-v_5-v_4-x$ is an induced $C_4$ in $B$, a contradiction. So $\deg_{B|T}(x)=5$ and condition (1) holds.
\end{proof}

\begin{lem}\label{5}
Let $H$ be a $C_4$-free graph and $v_1-v_2-v_3-v_4-v_5-v_1$ be an induced $C_5$ cycle in $H$. If $x\neq y\notin\{v_1,\dots,v_5\}$ are two vertices connected to the same three successive vertices of the cycle, then $xy\in E(H)$.
\end{lem}

\begin{proof}
Without loss of generality, assume that $\{x,y\}$ is $H$-complete to $\{v_1,v_2,v_3\}$. If $xy\notin E(H)$, then $v_1-y-v_3-x-v_1$ is an induced $C_4$ in $H$, a contradiction. It follows that $xy\in E(H)$.
\end{proof}
Now, we prove the main result.
\begin{thm}\label{7}
Let $B$ and $R$ be two $C_4$-free graphs on the same vertex set $V$ and assume that $G(B,R)$ is the complete graph. Then there exists an $B$-clique $X$, an $R$-clique $Y$ and a clique $Z$ in $B$ and $R$, such that $V=X\cup Y\cup Z$. Further, if $x\in Z$ then $x$ is one of the vertices of some double $C_5$ in $G(B,R)$.
\end{thm}
\begin{proof}
If $G(B,R)$ does not contains a double $C_5$, then by Theorem \ref{3} there exists an $B$-clique $X$, an $R$-clique $Y$ such that $V=X\cup Y$. By choosing $Z=\emptyset$, we are done. So assume that $G(B,R)$ contains a double $C_5$: $v_1-v_2-v_3-v_4-v_5-v_1$ with blue edges and red diagonals. If $V=\{v_1,\dots,v_5\}$, then by taking $X=\{v_1,v_2\}$, $Y=\{v_3,v_5\}$ and $Z=\{v_4\}$, we finish the proof. So assume that $V\setminus \{v_1,\dots,v_5\}\neq \emptyset$. We define the following sets:\\
\begin{center}
$M=\{p~|p\notin \{v_1,\dots,v_5\}~\mathrm{and}~pv_i\in E(B)~\mathrm{for~all}~ 1\leq i\leq 5\}$,\\
$N=\{p~|p\notin \{v_1,\dots,v_5\}~\mathrm{and}~pv_i\in E(R)~\mathrm{for~all}~ 1\leq i\leq 5\}$,\\
$A_j=\{p~|p\notin \{v_1,\dots,v_5\}~\mathrm{and}~pv_j~\mathrm{is~the~shared~edge~of}p\}$ for all $1\leq j\leq 5$.
\end{center}
By Lemma \ref{4}, we have $V\setminus\{v_1,\dots,v_5\}=M\cup N\cup A_1\cdots\cup A_5$.

If $p_1,p_2\in M$, then $p_1v_i\in E(B)$ and $p_2v_i\in E(B)$ for all $1\leq i\leq 5$. In particular, $p_1$ and $p_2$ are connected in $B$ to the same three successive vertices of the induced cycle $v_1-v_2-v_3-v_4-v_5-v_1$. Since $B$ is $C_4$-free, by Lemma \ref{5} we conclude that $p_1p_2\in E(B)$. It follows that $M$ is an $B$-clique. Similarly, if $p_1,p_2\in N$, then $p_1$ and $p_2$ are connected in $R$ to the same three successive vertices of the induced cycle $v_1-v_3-v_5-v_2-v_4-v_1$. By Lemma \ref{5}, we obtain that $N$ is an $R$-clique.\\
Let $p_1,p_2\in A_1$. Note that $\{p_1,p_2\}$ is $B$-complete to $\{v_1,v_2,v_5\}$ and $v_1,v_2,v_5$ are successive vertices in the induced cycle $v_1-v_2-v_3-v_4-v_5-v_1$ in $B$. By Lemma \ref{5}, we obtain that $p_1p_2\in E(B)$. Note also that $\{p_1,p_2\}$ is $R$-complete to $\{v_1,v_3,v_4\}$ and $v_1,v_3,v_4$ are successive vertices in the induced cycle $v_1-v_3-v_5-v_2-v_4-v_1$ in $R$. By Lemma \ref{5}, we obtain that $p_1p_2\in E(R)$. It follows that $A_1$ is a clique in $B$ and $R$. Similarly, $A_j$ is a clique in $B$ and $R$ for all $2\leq j\leq 5$.\\
A similar argument shows that $M\cup A_j$ is an $B$-clique for all $1\leq j\leq 5$ and $N\cup A_j$ is an $R$-clique for all $1\leq j\leq 5$. \\\\
\textbf{\emph{Claim:}} $M\cup A_1\cup A_2\cup \{v_1,v_2\}$ is a clique in $B$.\\\\
\emph{Proof of the claim:} If $x\in M\cup A_1$, then by the definition of $M$ and $A_1$, we obtain that $xv_1\in E(B)$. If $x\in A_2$, then $xv_2$ is the shared edge of $x$. Since $v_1$ is a neighbor (in $B\setminus R$) of $v_2$ in the induced cycle $v_1-v_2-v_3-v_4-v_5-v_1$, it follows that $xv_1\in E(B)$. So $\{v_1\}$ is $B$-complete to $M\cup A_1\cup A_2$. Similarly, $\{v_2\}$ is $B$-complete to $M\cup A_1\cup A_2$. We finish the proof of the claim if we show that $A_1$ is $B$-complete to $A_2$. Let $x_1\in A_1$ and $x_2\in A_2$. If $x_1x_2\in E(R)$, then $x_1-x_2-v_5-v_3-x_1$ is an induced $C_4$ in $R$, a contradiction. Since $G(B,R)$ is the complete graph it follows that $x_1x_2\in E(B)$. Thus, we proved the claim.\\\\
\textbf{\emph{Claim:}} $N\cup A_3\cup A_5\cup \{v_3,v_5\}$ is a clique in $R$.\\\\
\emph{Proof of the claim:} If $x\in N\cup A_3$, then by the definition of $N$ and $A_3$, we obtain that $xv_3\in E(R)$. If $x\in A_5$, then $xv_5$ is the shared edge of $x$. Since, $v_3$ is a neighbor (in $R\setminus B$) of $v_5$ in the induced cycle $v_1-v_3-v_5-v_2-v_4-v_1$, it follows that $xv_3\in E(R)$. So $\{v_3\}$ is $R$-complete to $N\cup A_3\cup A_5$. Similarly, $\{v_5\}$ is $R$-complete to $N\cup A_3\cup A_5$. We finish the proof of the claim if we show that $A_3$ is $R$-complete to $A_5$. Let $x_3\in A_3$ and $x_5\in A_5$. If $x_3x_5\in E(B)$, then $x_3-v_2-v_1-x_5-x_3$ is an induced $C_4$ in $B$, a contradiction. Since $G(B,R)$ is the complete graph it follows that $x_3x_5\in E(R)$. Thus, we proved the claim.\\\\
\textbf{\emph{Claim:}} $A_4\cup \{v_4\}$ is a clique in $B$ and in $R$.\\\\
\emph{Proof of the claim:} If $x\in A_4$, then $xv_4$ is the shared edge to $x$. So $xv_4$ is an edge in $B$ and $R$. Thus, the claim follows from the definition of $A_4$.\\\\
We set $$X=M\cup A_1\cup A_2\cup \{v_1,v_2\},~~~~~Y=N\cup A_3\cup A_5\cup \{v_3,v_5\},~~~~~Z=A_4\cup \{v_4\}.$$
By the above claims and Lemma \ref{4}, we obtain that $X$ is a clique in $B$, $Y$ is a clique in $R$ and $Z$ is a clique in $B$ and $R$, with $V=X\cup Y\cup Z$. \\

Let $x\in Z$. If $x=v_4$, then $v_1-v_2-v_3-v_4-v_5-v_1$ is a double $C_5$ that contains $v_4$ as a vertex. If $x\neq v_4$, then $x-v_5-v_1-v_2-v_3-x$ is a double $C_5$ in $G(B,R)$. Hence, if $x\in Z$ then $x$ is one of the vertices of some double $C_5$ in $G(B,R)$.
\end{proof}

As a corollary of Theorem \ref{7}, we obtain the following.

\begin{cor}
If $B$ and $R$ are two $C_4$-free graphs on the same vertex set $V$ then $$\omega(G(B,R))\leq \omega(B)+\omega(R)+\omega(H(B,R)).$$
\end{cor}

\begin{proof}
Let $T$ be a maximum clique in $G(B,R)$. So $B|T$ and $R|T$ are two $C_4$-free graphs on the same vertex set $T$ such that $G(B|T,R|T)$ is the complete graph. By Theorem \ref{7}, there exists an $B|T$-clique $X$, an $R|T$-clique $Y$ and a clique $Z$ in $B|T$ and $R|T$, such that $T=X\cup Y\cup Z$. So $$\omega(G(B,R))=|T|\leq |X|+|Y|+|Z|\leq \omega(B)+\omega(R)+\omega(H(B,R)).$$
\end{proof}

Also, we have the following additional corollary of Theorem \ref{7}.

\begin{cor}
If $B$ and $R$ are two $C_4$-free graphs on the same vertex set $V$ then $$\omega(G(B,R))\leq \omega(B)+\omega(R)+\frac{1}{2}\min(\omega(B),\omega(R)).$$
\end{cor}

\begin{proof}
Let $T$ be a maximum clique in $G(B,R)$. So $B|T$ and $R|T$ are two $C_4$-free graphs on the same vertex set $T$ such that $G(B|T,R|T)$ is the complete graph. If $G(B|T,R|T)$ does not contains a double $C_5$, then by Theorem \ref{3}, $T$ is the union of a clique in $B|T$ and a clique in $R|T$. It follows that $\omega(G(B,R))\leq \omega(B)+\omega(R)$ and the theorem holds. Assume that $G(B|T,R|T)$ contains a double $C_5$. Following the proof of Theorem \ref{7}, we may assume that $A_4$ is minimal so that $|A_4|\leq |A_i|$ for all $1\leq i\leq 5$. We obtain that $2|Z|=2|A_4|+2\leq|A_1|+|A_2|+|\{v_1,v_2\}|=|A_1\cup A_2\cup\{v_1,v_2\}|\leq \omega(B|T)\leq \omega(B)$. So $|Z|\leq \frac{1}{2}\omega(B)$. Similarly, $|Z|\leq \frac{1}{2}\omega(R)$. It follows that $$\omega(G(B,R))\leq \omega(B)+\omega(R)+\frac{1}{2}\min(\omega(B),\omega(R)).$$
\end{proof}

\end{document}